\documentclass[12pt]{article}
\usepackage{amsmath}
\usepackage{amssymb}
\usepackage{amsthm}
\usepackage{graphicx}
\usepackage{verbatim}
\usepackage{hyperref}
\usepackage{dsfont}
\numberwithin{equation}{section}

\newtheorem{thm}{Theorem}[section]
\newtheorem{cor}[thm]{Corollary}
\newtheorem{lem}[thm]{Lemma}

\begin{document}

\title{Volume growth and spectrum for general graph Laplacians}
\author{Matthew Folz\thanks{Department of Mathematics, The University of British Columbia, 1984 Mathematics Road, Vancouver, B.C., Canada, V6T 1Z2.  {\tt mfolz@math.ubc.ca}.  Research supported by a NSERC Alexander Graham Bell Canada Graduate Scholarship and a NSERC Michael Smith Foreign Study Supplement.}}
\maketitle

\begin{abstract}
\noindent  We prove estimates relating exponential or sub-exponential volume growth of weighted graphs to the bottom of the essential spectrum for general graph Laplacians.  The volume growth is computed with respect to a metric adapted to the Laplacian, and use of these metrics produces better results than those obtained from consideration of the graph metric only.  Conditions for absence of the essential spectrum are also discussed.  These estimates are shown to be optimal or near-optimal in certain settings and apply even if the Laplacian under consideration is an unbounded operator. \\

\noindent {\it Key words:} Bottom of spectrum, graph Laplacian, volume growth, random walk. \\
\noindent {\it AMS 2010 Subject Classification:} Primary: 47A10, 31C25, 60G50.
\end{abstract}

\section{Introduction} \label{S1}

Let $\Gamma=(G,E)$ be an unoriented, connected, countably infinite, locally finite graph.  We assume that $\Gamma$ has neither loops nor multiple edges.  We use $d$ to denote the graph metric on $\Gamma$; given $x,y\in G$, $d(x,y)$ is equal to the number of edges in a shortest (geodesic) path between $x$ and $y$. \\

We assume that $\Gamma$ is a weighted graph, so that associated with each $(x,y)\in G\times G$ is a nonnegative edge weight $\pi_{xy}$ which is symmetric ($\pi_{xy}=\pi_{yx}$ for $x,y\in G$) and satisfies $\pi_{xy}>0$ if and only if $\{x,y\}\in E$.  The edge weights define a measure on $G$ by setting $\pi_x := \pi(\{x\}) := \sum_{y\in G} \pi_{xy}$ for $x\in G$, and extending to all subsets of $G$ by countable additivity.  If $\pi(e)=1$ for all $e\in E$, we say that $\Gamma$ has the standard weights.  We denote weighted graphs by the pairing $(\Gamma, (\pi(e))_{e\in E})$; for brevity we write this as $(\Gamma,\pi)$. \\

Let $(\theta_x)_{x\in G}$ be an additional set of positive vertex weights.  We make the additional assumption that if $U\subset G$ is infinite and connected,

\begin{equation} \label{thetalb}
\sum_{x\in U} \theta_x = \infty.
\end{equation}
\\
The relevance of this condition will be discussed subsequently. \\

We will work on the (real) Hilbert space $L^2(\theta)$ with inner product

\begin{equation*}
\langle f,g\rangle_{L^2(\theta)} := \sum_{x\in G} f(x)g(x)\theta_x.
\end{equation*}
\\
The central object of our analysis will be the operator $\mathcal{L}_\theta$, which is defined pointwise for $f\in L^2(\theta)$ by

\begin{equation*}
(\mathcal{L}_\theta f)(x) := \frac{1}{\theta_x}\sum_{y\sim x}\pi_{xy}(f(y)-f(x)),
\end{equation*}
\\
and which has domain of definition $\mathcal{D}(\mathcal{L}_\theta) := \{f\in L^2(\theta):\mathcal{L}_\theta f\in L^2(\theta)\}$. \\

This operator is associated with the Dirichlet form $(\mathcal{E},\mathcal{D}(\mathcal{E}))$, where, for $f,g\in\mathcal{D}(\mathcal{E})$,

\begin{equation*}
\mathcal{E}(f,g) := \frac{1}{2}\sum_{x,y\in G} \pi_{xy}(f(y)-f(x))(g(y)-g(x)).
\end{equation*}
\\
We define $\mathcal{E}_1(f,f) := (\langle f,f\rangle_{L^2(\theta)} + \mathcal{E}(f,f))^{1/2}$ and let $\|\cdot\|_{H^1}$ denote the corresponding norm.  The domain $\mathcal{D}(\mathcal{E})$ is given by the closure of $C_c(G)$ in the norm $\|\cdot\|_{H^1}$.  Consequently $(\mathcal{E},\mathcal{D}(\mathcal{E}))$ is a closed, regular Dirichlet form.  See \cite{FOT} for more on the theory of Dirichlet forms. \\

{\bf Remark:}  The condition \eqref{thetalb} ensures that the generator associated with the regular Dirichlet form $(\mathcal{E},\mathcal{D}(\mathcal{E}))$ is the operator $\mathcal{L}_\theta$ on $\{f\in L^2(\theta):\mathcal{L}_\theta f \in L^2(\theta)\}$; see Theorem 5 of \cite{KL}. \\

The generator and Dirichlet form are linked by the Gauss-Green identity; for $f\in \mathcal{D}(\mathcal{L}_\theta)$ and $g\in \mathcal{D}(\mathcal{E})$,

\begin{equation} \label{GG}
\mathcal{E}(f,g) = -\langle\mathcal{L}_\theta f,g\rangle_{L^2(\theta)}.
\end{equation}
\\
In particular, $\langle -\mathcal{L}_\theta f,f\rangle_{L^2(\theta)} = \mathcal{E}(f,f) \geq 0$, and the operator $-\mathcal{L}_\theta$ is positive semidefinite. \\

In general, $\mathcal{L}_\theta$ is not a bounded operator on $L^2(\theta)$.  Define

\begin{equation*}
A(\Gamma,\pi,\theta) := \sup_{x\in G} \frac{\pi_x}{\theta_x}.
\end{equation*}
\\
Then $A(\Gamma,\pi,\theta) \leq \|\mathcal{L}_\theta\| \leq 2A(\Gamma,\pi,\theta)$, and $\mathcal{L}_\theta$ is bounded if and only if $A(\Gamma,\pi,\theta)<\infty$.  See Lemma 1 of \cite{D2} for a proof.  For simplicity of notation, we will write $\|\cdot\|$ for the operator norm on $L^2(\theta)$.  Additionally, for this quantity and subsequent quantities which depend on the edge weights $(\pi(e))_{e\in E}$, if the edge weights under consideration are the standard weights, we will simply write omit reference to the edge weights and write (for example) $A(\Gamma,\theta)$. \\

Each choice of $(\theta_x)_{x\in G}$ corresponds to a continuous-time simple random walk on $(\Gamma,\pi)$ with infinitesimal generator $\mathcal{L}_\theta$, which we denote by $(X^\theta_t)_{t\geq 0}$.  This process is strong Markov; started at a vertex $x\in G$, it waits an exponentially distributed time with parameter $\pi_x/\theta_x$ and then jumps to a neighbour with jump probabilities $P(x,y) := \pi_{xy}/\pi_x$.  In particular, modifying the vertex weights $(\theta_x)_{x\in G}$ induces a time-change of the random walk; the jump times change but not the jump probabilities. \\

We single out two choices of the vertex weights $(\theta_x)_{x\in G}$.  The first is the choice $(\pi_x)_{x\in G}$.  The corresponding random walk is referred to as the constant-speed random walk (because it jumps at exponentially distributed times with mean $1$), and has infinitesimal generator

\begin{equation*}
(\mathcal{L}_\pi f)(x) := \frac{1}{\pi_x}\sum_{y\sim x}\pi_{xy}(f(y)-f(x)).
\end{equation*}
\\
This operator is sometimes called the probabilistic or normalized Laplacian; it is a bounded operator on $L^2(\pi)$ satisfying $1\leq \|\mathcal{L}_\pi\| \leq 2$. \\

The second choice are the vertex weights $(\mathds{1}_x)_{x\in G}$, where ${\mathds{1}}_x := 1$ for each $x\in G$.  The corresponding random walk is called the variable-speed random walk (because it jumps at exponentially distributed times with parameter $\pi_x$), and has infinitesimal generator 

\begin{equation*}
(\mathcal{L}_{\mathds{1}}f)(x) := \sum_{y\sim x} \pi_{xy}(f(y)-f(x)).
\end{equation*}
\\
This operator is sometimes called the physical or unnormalized Laplacian, and is a bounded operator on $L^2(\mathds{1})$ if and only if the weights $(\pi_x)_{x\in G}$ are bounded above. \\

Associated with the process $(X^\theta_t)_{t\geq 0}$ is the semigroup $(P^\theta_t)_{t\geq 0}$, where $P^\theta_t = \exp(t\mathcal{L}_\theta)$ (or, equivalently, $P^\theta_tf(x) := \mathbb{E}^xf(X_t)$).  The semigroup has a density $p^\theta_t(x,y)$ with respect to the measure $\theta$ given by

\begin{equation*}
p_t(x,y) := \frac{1}{\theta_y}\mathbb{P}^x(X^\theta_t=y).
\end{equation*}
\\
This density is also called the heat kernel of $(X^\theta_t)_{\geq 0}$. \\

\subsection{Metrics and Volume}

{\bf Definition:} The metric $\rho$ is adapted to the operator $\mathcal{L}_\theta$ on $(\Gamma,\pi)$ if for all $x\in G$, 

\begin{equation}\label{adaptedness}
\frac{1}{\theta_x}\sum_{y\sim x}\pi_{xy}\rho^2(x,y) \leq 1,
\end{equation}
\\
and there exists $c_\rho>0$ such that $\rho(x,y) \leq c_\rho$ whenever $x\sim y$. \\

The following metrics are useful examples of adapted metrics: \\

{\bf 1.} On any weighted graph, the graph metric $d$ is adapted to the operator $\mathcal{L}_\pi$.  More generally, if $\mathcal{L}_\theta$ is bounded, (so that $A(\Gamma,\pi,\theta)<\infty$), then the metric $\frac{1}{\sqrt{A(\Gamma,\pi,\theta)}}d$ is adapted to $\mathcal{L}_\theta$. \\

No fixed multiple of the graph metric can be adapted to $\mathcal{L}_\theta$ if $\mathcal{L}_\theta$ is an unbounded operator.  In this case, since there is no uniform upper bound on $\pi_x/\theta_x$, given a metric $\rho$ adapted to $\mathcal{L}_\theta$, for any $\varepsilon>0$, there exist adjacent vertices $x$ and $y$ with $\rho(x,y)\leq \varepsilon$. \\

{\bf 2.}  On any weighted graph with vertex degrees uniformly bounded above by $D$, the metric $\frac{1}{\sqrt{D}}d_E$ is adapted to the operator $\mathcal{L}_{\mathds{1}}$, where for $x,y\in G$,

\begin{equation*}
d_E(x,y) := \inf\left\{\sum_{e\in\gamma} 1\wedge \pi(e)^{-1/2}:\text{$\gamma$ is a path joining $x$ and $y$}\right\}.
\end{equation*}

If the graph is $k-$regular, then at each $x\in G$ there is equality in \eqref{adaptedness}. \\

{\bf 3.}  On any weighted graph, the metric $d_V$ is adapted to the operator $\mathcal{L}_\theta$, where for $x,y\in G$, 

\begin{equation*}
d_V(x,y) := \inf\left\{\sum_{e\in\gamma} 1\wedge c(e):\text{$\gamma$ is a path joining $x$ and $y$}\right\},
\end{equation*}
\\
and for $e := \{u,v\}$,

\begin{equation*}
c(e) := \left(\frac{\theta_u}{\pi_u}\wedge\frac{\theta_v}{\pi_v}\right)^{1/2}.
\end{equation*}
\\
The condition of adaptedness has appeared in several recent papers; see \cite{F1} and \cite{F2} by the author, where adapted metrics are used to obtain heat kernel estimates and criteria for stochastic completeness, \cite{GHM}, where adapted metrics are used to obtain criteria for stochastic completeness, and \cite{H}, which deals with uniqueness of solutions to the heat equation.  This condition was communicated to the author by A. Grigor'yan, and it has some similarities with certain distance functions on graphs considered by Davies in \cite{D1}.  The condition of adaptedness, as well as its relation to certain metrics on graphs considered by other authors, is discussed in greater detail in Section 2 of \cite{F2}. \\

Let $\rho$ be a metric adapted to the operator $\mathcal{L}_\theta$.  Write

\begin{equation*}
B_\rho(x_0,r) := \{x\in G:\rho(x_0,x)\leq r\}
\end{equation*}
\\
for the closed ball of radius $r$ in this metric, and

\begin{equation*}
V_\rho(x_0,r) := \sum_{x\in B_\rho(x_0,r)} \theta_x
\end{equation*}
\\
for the volume of this ball with respect to the measure $(\theta_x)_{x\in G}$. \\

We define the exponential volume growth of $(\Gamma,\pi)$ with respect to the measure $\theta$ and the adapted metric $\rho$ by

\begin{equation}\label{expVG}
VG(\Gamma,\pi,\theta,\rho) := \limsup_{r\to\infty} \frac{1}{r}\log V_\rho(x_0,r).
\end{equation}
\\
An application of the triangle inequality shows that the right hand side of \eqref{expVG} is independent of the choice of $x_0\in G$. \\ 

If $VG(\Gamma,\pi,\theta,\rho)=0$ (respectively $VG(\Gamma,\pi,\theta,\rho)\in (0,\infty)$, $VG(\Gamma,\pi,\theta,\rho)=\infty$), we say that $\Gamma$ has subexponential (respectively exponential, superexponential) volume growth with respect to the measure $\theta$ and the metric $\rho$. \\

\subsection{Spectra}

Let $\sigma(\Gamma,\pi,\theta)$ denote the spectrum of $-\mathcal{L}_\theta$ on $(\Gamma,\pi)$.  As a positive semidefinite self-adjoint operator, the bottom of the spectrum of $-\mathcal{L}_\theta$ is nonnegative, and we denote it by $\lambda_0(\Gamma,\pi,\theta)$. \\

We have the following variational expression for the bottom of the spectrum:

\begin{equation}\label{variational}
\lambda_0(\Gamma,\theta) := \inf\{\mathcal{E}(f,f):f\in \mathcal{D}(\mathcal{E}), \ \|f\|^2_{L^2(\theta)} = 1\}.
\end{equation}
\\
For $V\subset G$, we define the Dirichlet Laplacian on $V$ on $\mathcal{D}(\mathcal{L}_\theta)\cap \{f\in L^2(\theta):f|_{G\setminus V} = 0\}$ by \\

\begin{equation*}
\mathcal{L}^V_\theta f(x) := \begin{cases} \mathcal{L}_\theta f(x) &\text{if }x\in V, \\ 0 &\text{if }x\in G\setminus V. \end{cases}
\end{equation*}
\\
Let $\lambda^V_0(\Gamma,\pi,\theta)$ denote the bottom of the spectrum of the Dirichlet Laplacian on $V$.  If $(G_k)_{k\in\mathbb{Z}_+}$ are finite sets which increase to $G$, then we have that as $k\to\infty$,

\begin{equation*}
\lambda^{G_k}_0(\Gamma,\pi,\theta) \downarrow \lambda_0(\Gamma,\pi,\theta).
\end{equation*}
\\
We use the following identity together with \eqref{variational} to obtain that for $f\in L^2(\theta)$,

\begin{equation*}
\frac{d}{dt}\|P^\theta_t f\|^2_{L^2(\theta)} = -2\langle \mathcal{L}_\theta P^\theta_t f,P^\theta_t f\rangle_{L^2(\theta)} \leq -2\lambda_0(\Gamma,\pi,\theta)\| P^\theta_t f\|^2,
\end{equation*}
\\
and this differential inequality together with $P^\theta_0=I$ gives $\|P^\theta_t\| \leq \exp(-\lambda_0(\Gamma,\pi,\theta)t)$.  On the other hand, since $\mathcal{L}^V_\theta$ and $P^\theta_t$ commute, if $(G_k)_{k\in\mathbb{Z}_+}$ are finite sets which increase to $G$
and $\lambda_k := \lambda^{G_k}_0(\Gamma,\theta,\pi)$ has corresponding eigenfunction $f_k$, it follows that for all $k\in\mathbb{Z}_+$,

\begin{equation*}
\frac{d}{dt}\|P^\theta_t f_k\|^2_{L^2(\theta)} = -2\langle \mathcal{L}^{G_k}_\theta P^\theta_t f_k,P^\theta_t f_k\rangle_{L^2(\theta)} = -2\lambda_k \|P^\theta_t f_k\|^2_{L^2(\theta)}.
\end{equation*}
\\
Solving the differential equation gives $\|P^\theta_t\| \geq \exp(-\lambda_k t)$, from which we conclude that $\|P^\theta_t\| = \exp(-\lambda_0(\Gamma,\pi,\theta)t)$. \\

In \cite{KLVW} a closely related bound for the heat kernel is proved; for all $x,y\in G$,

\begin{equation*}
\lim_{t\to\infty}\frac{1}{t}\log p^\theta_t(x,y) \to -\lambda_0(\Gamma,\pi,\theta).
\end{equation*}
\\
Our results concern the bottom of the essential spectrum of $-\mathcal{L}_\theta$.  The essential spectrum $\sigma_{ess}(\Gamma,\pi,\theta)$ consists of points of $\sigma(\Gamma,\pi,\theta)$ which are either accumulation points of $\sigma(\Gamma,\pi,\theta)$ or which are discrete eigenvalues of $-\mathcal{L}_\theta$ with infinite multiplicity.  The discrete spectrum of $-\mathcal{L}_\theta$ is defined by $\sigma_{disc}(\Gamma,\pi,\theta) = \sigma(\Gamma,\pi,\theta)\setminus \sigma_{ess}(\Gamma,\pi,\theta)$. \\ 

We denote the bottom of the essential spectrum by $\lambda^{ess}_0(\Gamma,\pi,\theta)$.  Clearly, $\lambda_0(\Gamma,\pi,\theta) \leq \lambda^{ess}_0(\Gamma,\pi,\theta)$.  The essential spectrum is invariant under finite perturbations (see Lemma 1 of \cite{Fu2}), and by standard techniques of spectral theory, if $(G_k)_{k\in\mathbb{Z}_+}$ is any sequence of finite subsets of $G$ which increase to $G$, then we have that

\begin{equation*}
\lambda^{ess}_0(\Gamma,\pi,\theta) = \lim_{k\to\infty} \lambda^{G\setminus G_k}_0(\Gamma,\pi,\theta).
\end{equation*}
\\
Our main results are estimates relating the volume growth in adapted metrics to the bottom of the essential spectrum.  In the setting of Riemannian manifolds, Brooks proved the following result relating volume growth in the Riemannian metric with the bottom of the spectrum:

\begin{thm}
(Brooks, \cite{Br1}) Let $M$ be a smooth, complete, non-compact Riemannian manifold with Riemannian metric $\rho_M$ and Riemannian volume $V_M$.  Let $\lambda^{ess}_0(M)$ denote the bottom of the essential spectrum of the Laplace-Beltrami operator $-\Delta_M$ on $L^2(M)$.  Define, for some $x_0\in M$,

\begin{equation*}
\mu_M := \limsup_{r\to\infty} \frac{1}{r}\log V_M(B_{\rho_M}(x_0,r)).
\end{equation*}
\\
If $M$ has infinite volume, then

\begin{equation*}
\lambda_0^{ess}(M) \leq \frac{1}{4}\mu^2_M.
\end{equation*}
\end{thm}

A similar estimate was subsequently established in the more general setting of strongly local Dirichlet spaces in \cite{St}. \\

In the setting of unweighted graphs (where the associated Dirichlet form is nonlocal), Fujiwara proved an analogue of these results for the spectrum of the operator $-\mathcal{L}_\pi$:

\begin{thm}
(Fujiwara, \cite{Fu1})  If $\Gamma$ is an unoriented, connected, countably infinite, locally finite graph with the standard weights, and $\mu_\Gamma := VG(\Gamma,\pi,d)$, then

\begin{equation*}
\lambda^{ess}_0(\Gamma,\pi) \leq 1+\frac{2\exp(\mu_\Gamma/2)}{1+\exp(\mu_\Gamma)}.
\end{equation*}
\end{thm}

While the Laplace-Beltrami operator $\Delta_M$ considered by Brooks may be unbounded, the operator $\mathcal{L}_\pi$ considered by Fujiwara is bounded.  In this setting, there is no need to consider metrics besides the graph metric, and the boundedness of $\mathcal{L}_\pi$ rules out various interesting behaviors, such as absence of essential spectrum, or the possibility of stochastic incompleteness. \\

We have two main results for the bottom of the essential spectrum.  For simplicity, we fix an adapted metric $\rho$ and set $\mu := VG(\Gamma,\pi,\theta,\rho)$. \\

The first result has no restrictions on the adapted metric $\rho$ and is useful even when $\mathcal{L}_\theta$ is unbounded.  

\begin{thm}\label{thm2}
The following general estimate for the bottom of the spectrum holds:
\begin{equation*}
\lambda^{ess}_0(\Gamma,\pi,\theta) \leq \frac{1}{8}\mu^2.
\end{equation*}
\end{thm}

The exponent $2$ cannot be reduced; see Example 1 in Section 4. \\

The second result generalizes results of Fujiwara in \cite{F1} and is useful when $\mathcal{L}_\theta$ is bounded.

\begin{thm}\label{thm1}
Suppose that there exist positive constants $m,M$ such that $m\leq \rho(x,y)\leq M$ whenever $x\sim y$.  Then

\begin{equation*}
\lambda^{ess}_0(\Gamma,\pi,\theta) \leq \frac{1}{m^2}\frac{(1-\exp(\frac{1}{2}M\mu))^2}{1+\exp(M\mu)}.
\end{equation*}
\end{thm}

As discussed previously, such metrics always exist (and can be taken to be a multiple of the graph metric) when $\mathcal{L}_\theta$ is bounded.  On a $k$-regular tree with the standard weights, there is equality for $\mathcal{L}_\pi$ when using the graph metric $d$, or for $\mathcal{L}_{\mathds{1}}$ when using the metric $\frac{1}{\sqrt{k}}d$. \\

The former result has the following two immediate corollaries:

\begin{cor}\label{cor1}
If there exists an adapted metric $\rho$ such that $VG(\Gamma,\pi,\theta,\rho)=0$, then $\lambda^{ess}_0(\Gamma,\pi,\theta)=0$. 
\end{cor}

The converse of this statement is false; a counterexample is provided by the operator $\mathcal{L}_\pi$ on the Cayley graph of any solvable group with exponential growth (see \cite{DK}), such as the lamplighter group.

\begin{cor}\label{cor2}
If there exists an adapted metric $\rho$ such that $VG(\Gamma,\pi,\theta,\rho)<\infty$, the essential spectrum is nonempty.  
\end{cor}

This result cannot be improved upon in the following sense: For every $\varepsilon>0$, there exists a graph $(\Gamma,\pi)$ for which the essential spectrum of $\mathcal{L}_{\mathds{1}}$ is empty, but there exists an adapted metric $\rho$ such that

\begin{equation*}
\mu_{\rho} := \limsup_{r\to\infty} \frac{1}{r^{1+\varepsilon}} \log V_{\rho}(x_0,r) < \infty.
\end{equation*}
\\
See Example 2 in Section 4. \\

The structure of this paper is as follows.  We prove Theorem~\ref{thm1} and Theorem~\ref{thm2} in Section 2; this is accomplished by using techniques of Brooks and Fujiwara together with the adapted metrics defined earlier.  In particular, the Dirichlet forms of certain functions of these metrics can be controlled by the $L^2(\theta)$ norm.  In Section 3, we discuss the phenomena of absence of essential spectrum and its relation to volume growth, and in Section 4, we provide various examples showing the sharpness of our results, as well as settings where the graph metric gives very poor results and use of adapted metrics seems to be essential.

\section{Proofs}

We fix a weighted graph $(\Gamma,\pi)$ and a set of positive  vertex weights $(\theta_x)_{x\in G}$ satisfying \eqref{thetalb}.  We also fix a metric $\rho$ which is adapted to the operator $\mathcal{L}_\theta$, and set $\rho_{x_0}(x) := \rho(x_0,x)$ for $x_0\in G$.  By the triangle inequality, $|\rho_{x_0}(x)-\rho_{x_0}(y)|\leq \rho(x,y)$.  As before, we set $\mu := VG(\Gamma,\pi,\theta,\rho)$. \\

We note that it suffices to assume that $\mu < \infty$.  If $\mu = \infty$, then the conclusion of Theorem~\ref{thm1} is $\lambda^{ess}_0(\Gamma,\pi,\theta)\leq \infty$, which is trivial.  If the hypotheses of Theorem~\ref{thm2} are satisfied (with $m < \rho(x,y)$ whenever $x\sim y$) and $\mu=\infty$, then the conclusion of Theorem~\ref{thm2} is 

\begin{equation*}
\lambda^{ess}_0(\Gamma,\pi,\theta)\leq \frac{1}{m^2}.
\end{equation*}
\\
On the other hand, this estimate may be proven directly by noting that for $x\in G$, $\|\theta_x^{-1/2}{\bf 1}_x\|_{L^2(\theta)} = 1$, and 

\begin{equation*}
\mathcal{E}(\theta_x^{-1}{\bf 1}_x,\theta_x^{-1}{\bf 1}_x) = \frac{\pi_x}{\theta_x} \leq \frac{1}{m^2},
\end{equation*}
\\
where the last inequality follows from adaptedness of $\rho$ and the inequality $m<\rho(x,y)$ for $x\sim y$. \\


In particular, the assumption $\mu<\infty$ implies that for all $x_0\in G$ and $r\geq 0$, the balls $B_\rho(x_0,r)$ have finite volume (i.e., for all $x_0\in G$ and $r\geq 0$, $V_\rho(x_0,r)<\infty)$.  Combining this with \eqref{thetalb}, we see that for $x_0\in G$ and $r\geq 0$, the balls $B_\rho(x_0,r)$ contain only finitely many points, as well. \\
 
The proof follows the general techniques of \cite{Br1} and \cite{F1}.

\begin{lem}
If $\mu < 2\alpha$, then $\langle \exp(-\alpha\rho_{x_0}),\exp(-\alpha\rho_{x_0})\rangle_{L^2(\theta)}<\infty$.
\end{lem}
\begin{proof}
We estimate 

\begin{align*}
\langle \exp(-\alpha\rho_{x_0}),\exp(-\alpha\rho_{x_0})\rangle_{L^2(\theta)} &= \sum_{x\in G}\exp(-2\alpha\rho_{x_0}(x))\theta_x \\
&\leq \theta_{x_0}+\sum^\infty_{r=0}\sum_{\rho_{x_0}(x)\in (r,r+1]} \exp(-2\alpha r)\theta_x \\
&= \theta_{x_0} + \sum^\infty_{r=0} (V_\rho(x_0,r+1)-V_\rho(x_0,r))\exp(-2\alpha r) \\
&=(\exp(2\alpha)-1)\sum^\infty_{r=1} V_\rho(x_0,r)\exp(-2\alpha r).
\end{align*}
\\
The last expression is finite by the definition of $\mu$ and comparison with a geometric series.
\end{proof}

For $j\in\mathbb{Z}_+$, we define a sequence of `tent functions' by

\begin{equation*}
h_j(x) := \begin{cases} \alpha \rho_{x_0}(x) &\text{if }\rho_{x_0}(x) \leq j, \\ 2\alpha j-\alpha\rho_{x_0}(x) &\text{if }\rho_{x_0}(x) > j, \end{cases}
\end{equation*}
\\
and set $f_j(x) := \exp(h_j)$. \\

We have the following estimate:

\begin{lem}
If $x\sim y$, for all $j\in\mathbb{Z}_+$,

\begin{equation}\label{ineq1}
(f_j(y)-f_j(x))^2 \leq \frac{(1-\exp(\alpha\rho(x,y)))^2}{1+\exp(2\alpha\rho(x,y))} (f^2_j(x)+f^2_j(y)).
\end{equation}
\end{lem}
\begin{proof}
Fix $j\in\mathbb{Z}_+$.  We prove first that for all $x\sim y$ and all $j\in\mathbb{Z}_+$, we have

\begin{equation}\label{ineq2}
(f_j(y)-f_j(x))^2 \leq \frac{(1-\exp(\alpha|\rho_{x_0}(y)-\rho_{x_0}(x)|))^2}{1+\exp(2\alpha|\rho_{x_0}(y)-\rho_{x_0}(x)|)} (f^2_j(x)+f^2_j(y)).
\end{equation}
\\
We may assume without loss of generality that $\rho_{x_0}(y)\geq \rho_{x_0}(x)$.  The proof proceeds by checking two cases. \\

{\it Case 1:} $j\geq \rho_{x_0}(y)\geq \rho_{x_0}(x)$ or $j\leq \rho_{x_0}(x)\leq \rho_{x_0}(y)$. \\

In this case, one verifies directly that there is equality in \eqref{ineq2}. \\

{\it Case 2:} $\rho_{x_0}(x)<j<\rho_{x_0}(y)$. \\

In this case, we begin by noting that the function $t\to \frac{(1-\exp(t))^2}{1+\exp(2t)}$ is increasing on the positive real line.  Using this, we get that

\begin{align*}
\frac{(f_j(y)-f_j(x))^2}{f^2_j(x)+f^2_j(y)} &= \frac{(1-\exp(\alpha(2j-(\rho_{x_0}(x)+\rho_{x_0}(y)))))^2}{1+\exp(2\alpha(2j-(\rho_{x_0}(x)+\rho_{x_0}(y))))} \\
&\leq \frac{(1-\exp(\alpha(\rho_{x_0}(y)-\rho_{x_0}(x)))^2}{1+\exp(2\alpha(\rho_{x_0}(y)-\rho_{x_0}(x)))},
\end{align*}
\\
where the inequality follows since $2j-(\rho_{x_0}(y)+\rho_{x_0}(x)) \leq \rho_{x_0}(y)-\rho_{x_0}(x)$ is equivalent to $j\leq \rho_{x_0}(y)$, which is true by hypothesis. \\
\end{proof}

Since the function $t\to \frac{(1-\exp(t))^2}{1+\exp(2t)}$ is increasing on the positive real line and $|\rho_{x_0}(x)-\rho_{x_0}(y)|\leq \rho(x,y)$, we have from \eqref{ineq2} that for all $x\sim y$ and all $j\in\mathbb{Z}_+$,

\begin{align*}
(f_j(y)-f_j(x))^2 &\leq \frac{(1-\exp(\alpha|\rho_{x_0}(y)-\rho_{x_0}(x)|)^2}{1+\exp(2\alpha|\rho_{x_0}(y)-\rho_{x_0}(x)|)} (f^2_j(x)+f^2_j(y)) \\
&\leq \frac{(1-\exp(\alpha\rho(x,y)))^2}{1+\exp(2\alpha\rho(x,y))} (f^2_j(x)+f^2_j(y)),
\end{align*}
\\
which completes the proof. \\

We will now use \eqref{ineq1} to prove two new bounds:

\begin{cor}\label{corThm1}
Suppose there exist constants $m,M>0$ such that $m \leq \rho(x,y)\leq M$ whenever $x\sim y$.  Then for all $x\sim y$ and $j\in\mathbb{Z}_+$,

\begin{equation*}
(f_j(y)-f_j(x))^2 \leq \frac{\rho^2(x,y)}{m^2}\frac{(1-\exp(\alpha M))^2}{1+\exp(2\alpha M)} (f^2_j(x)+f^2_j(y)).
\end{equation*}
\end{cor}
\begin{proof}
This follows immediately from $\frac{\rho^2(x,y)}{m^2}\geq 1$ and the fact that $t\to \frac{(1-\exp(t))^2}{1+\exp(2t)}$ is increasing on the positive real line.
\end{proof}
\begin{cor}\label{corThm2}
For $x\sim y$ and any $j\in\mathbb{Z}_+$,

\begin{equation*}
(f_j(y)-f_j(x))^2 \leq \rho^2(x,y)\frac{\alpha^2}{2}(f^2_j(x)+f^2_j(y)).
\end{equation*}
\end{cor}
\begin{proof}
This follows immediately from the estimate $\frac{(1-\exp(t))^2}{1+\exp(2t)}\leq \frac{t^2}{2}$, which is valid for all $t\geq 0$.
\end{proof}

These estimates give us control of $\mathcal{E}(f_j,f_j)$, as follows:

\begin{lem}\label{DFcontrol}
Suppose that for $x\sim y$, $(f(y)-f(x))^2 \leq C\rho^2(x,y)(f^2(x)+f^2(y))$.  Then

\begin{equation*}
\mathcal{E}(f,f) \leq C\|f\|^2_{L^2(\theta)}.
\end{equation*}
\end{lem}
\begin{proof}
Using adaptedness of $\rho$, we have

\begin{align*}
\mathcal{E}(f,f) &:= \frac{1}{2}\sum_{x,y\in G}\pi_{xy}(f(y)-f(x))^2 \\
&\leq \frac{C}{2}\sum_{x\in G}\sum_{y\sim x}\pi_{xy}\rho^2(x,y)(f^2(x)+f^2(y)) \\
&= C \sum_{x\in G}\left(\frac{1}{\theta_x}\sum_{y\sim x} \pi_{xy}\rho^2(x,y)\right)f^2(x)\theta_x \\
&\leq C\|f\|^2_{L^2(\theta)}.
\end{align*}
\end{proof}

Combining Corollary~\ref{corThm1} and Corollary~\ref{corThm2} with Lemma~\ref{DFcontrol}, we get:

\begin{lem}
Suppose there exist constants $m,M>0$ such that $m \leq \rho(x,y)\leq M$ whenever $x\sim y$.  For all $j\in\mathbb{Z}_+$,

\begin{equation*}\label{thm1DF}
\mathcal{E}(f_j,f_j) \leq \frac{1}{m^2}\frac{(1-\exp(\alpha M))^2}{1+\exp(2\alpha M)}\|f_j\|^2_{L^2(\theta)}.
\end{equation*}
\end{lem}
\begin{lem}
For all $j\in\mathbb{Z}_+$,

\begin{equation*}\label{thm2DF}
\mathcal{E}(f_j,f_j) \leq \frac{\alpha^2}{2}\|f_j\|^2_{L^2(\theta)}.
\end{equation*}
\end{lem}

Now, to combine the previous results, we assume that there exists a positive, increasing, continuous function $t \to I_\rho(t)$ such that 

\begin{equation} \label{DFcombined}
\mathcal{E}(f_j,f_j) \leq I_\rho(\alpha)\|f_j\|^2_{L^2(\theta)}.
\end{equation}
\\
Let $K$ be a finite set, and define $M_K := \max\{\rho(x_0,x):x\in K\}$ and $g_j := f_j(1-{\bf 1}_K)$.  Note that by construction, $B_\rho(x_0,M_K)\supset K$. \\

\begin{lem}\label{lem2Ng}
For all $j\in\mathbb{Z}_+$, $\langle g_j,g_j\rangle_{L^2(\theta)} < \infty$, and

\begin{equation*}
\lim_{j\to \infty} \langle g_j,g_j\rangle_{L^2(\theta)} = \infty.
\end{equation*}
\end{lem}
\begin{proof}
First, for $j\in\mathbb{Z}_+$,

\begin{align*}
\langle g_j,g_j\rangle_{L^2(\theta)} &= \sum_{x\in B_\rho(x_0,j)} g^2_j(x)\theta_x+\sum_{x\in G\setminus B_\rho(x_0,j)} g^2_j(x)\theta_x \\
&\leq \exp(2\alpha j)V_\rho(x_0,j)+\exp(4\alpha j)\langle \exp(-\alpha\rho_{x_0}),\exp(-\alpha\rho_{x_0})\rangle_{L^2(\theta)}.
\end{align*}
\\
For the second part, if $j>M_K$, then since $g_j \geq 1$ on $B_\rho(x_0,j)\setminus B_\rho(x_0,M_K)$, it follows that 

\begin{equation*}
\langle g_j,g_j\rangle_{L^2(\theta)} \geq V_\rho(x_0,j)-V_\rho(x_0,M_K),
\end{equation*}
\\
and since $V_\rho(x_0,M_K)<\infty$ and $V_\rho(x_0,j)\uparrow \infty$ as $j\to\infty$ (using \eqref{thetalb}), the result follows.
\end{proof}

Next, we estimate the Dirichlet forms $\mathcal{E}(g_j,g_j)$ as follows:

\begin{lem}\label{lemDFg}
There exists a constant $c(\rho,\alpha,K)$ such that for $j\in\mathbb{Z}_+$,

\begin{equation*}
\mathcal{E}(g_j,g_j) \leq c(\rho,\alpha,K)+I_\rho(\alpha)\langle g_j,g_j\rangle_{L^2(\theta)}.
\end{equation*}
\\
Additionally, $g_j\in \mathcal{D}(\mathcal{E})$ for each $j\in\mathbb{Z}_+$.
\end{lem}
\begin{proof}
First, we note that $g_j = f_j$ on $G\setminus B_\rho(x_0,M_K)$, and if $r\geq M_K$, on $B_\rho(x_0,r)$, we have $0\leq g_j \leq f_j \leq \exp(\alpha r)$. \\

Consequently, we have the crude estimates

\begin{align}
\nonumber \mathcal{E}(g_j,g_j) &\leq \mathcal{E}(f_j,f_j) + \pi(B_\rho(x_0,M_K+c_\rho))\sup_{x\in B_\rho(x_0,M_K+c_\rho)} |g^2_j(x)| \\
&\leq \mathcal{E}(f_j,f_j) + \pi(B_\rho(x_0,M_K+c_\rho))\exp(2\alpha(M_K+c_\rho)). \label{est1}
\end{align}
\\
The reason for using the larger distance $M_K+c_\rho$ is that all edges with at least one end in $B_\rho(x_0,M_K)$ have both ends in $B_\rho(x_0,M_K+c_\rho)$.  We also have the simple estimate

\begin{equation} \label{est2}
\langle f_j,f_j\rangle_{L^2(\theta)} \leq \langle g_j,g_j\rangle_{L^2(\theta)} +\exp(2\alpha M_K)V_\rho(x_0,M_K).
\end{equation}
\\
Combining \eqref{est1} with \eqref{DFcombined} and \eqref{est2} gives the desired result, with

\begin{equation*}
c(\rho,\alpha,K) = \exp(2\alpha M_K)(\pi(B_\rho(x_0,M_K+c_\rho))\exp(2\alpha c_\rho)+I_\rho(\alpha)V_\rho(x_0,M_K)).
\end{equation*}
\\
Next, we note that for each $j\in\mathbb{Z}_+$ and $\varepsilon>0$, if  $R^\varepsilon_j \geq j \vee \left(2j-\frac{\log \varepsilon}{\alpha}\right)$, then on $G\setminus B_\rho(x_0,R^\varepsilon_j)$, $g_k \leq \varepsilon$.  Since $B_\rho(x_0,R_\varepsilon)$ is finite, $g^\varepsilon_k := (g-\varepsilon)_+\in C_c(G)$.  Now, since $g^\varepsilon_k \uparrow g_k$ as $\varepsilon\downarrow 0$.  Since $\mathcal{E}(g^\varepsilon_k,g^\varepsilon_k) \leq \mathcal{E}(g_k,g_k)$ and $\|g_k\|_{H^1} < \infty$, we conclude that $\|g^\varepsilon_k-g_k\|_{H^1} \to 0$ as $\varepsilon\downarrow 0$, and hence $g_k \in \mathcal{D}(\mathcal{E})$.
\end{proof}

Finally, we can show the following result:

\begin{thm}\label{endthm}
Assume that $\rho$ is such that \eqref{DFcombined} holds.  Then

\begin{equation*}
\lambda^{ess}_0(\Gamma,\pi,\theta) \leq I_\rho(\mu/2).
\end{equation*}
\end{thm}
\begin{proof}
We argue by contradiction.  Suppose that 

\begin{equation*}
I_\rho(\mu/2) \leq \lambda^{ess}_0(\Gamma,\pi,\theta).
\end{equation*}
\\
By definition, there exists a finite subset $K$ such that $I_\rho(\mu/2) < \lambda^{G\setminus K}_0(\Gamma,\pi,\theta)$.  Since $t\to I_\rho(t)$ is increasing and continuous, we can find a constant $\alpha$ such that $\mu<2\alpha$ and such that

\begin{equation*}
I_\rho(\mu/2) < I_\rho(\alpha) < \lambda^{G\setminus K}_0(\Gamma,\pi,\theta).
\end{equation*}
\\
On the other hand, from Lemma~\ref{lemDFg}

\begin{equation*}
\frac{\mathcal{E}(g_j,g_j)}{\langle g_j,g_j\rangle_{L^2(\theta)}} \leq \frac{c(\rho,\alpha,K)}{\langle g_j,g_j\rangle_{L^2(\theta)}}+I_\rho(\alpha),
\end{equation*}
\\
so letting $j\to\infty$ and using Lemma~\ref{lem2Ng} shows if $j_0$ is chosen sufficiently large, then

\begin{equation*}
\frac{\mathcal{E}(g_{j_0},g_{j_0})}{\langle g_{j_0},g_{j_0}\rangle_{L^2(\theta)}} \leq I_\rho(\alpha) < \lambda^{G\setminus K}_0(\Gamma,\pi,\theta).
\end{equation*}
\\
This is a contradiction, since $g_{j_0}|_K=0$, and

\begin{equation*}
\lambda^{G\setminus K}_0(\Gamma,\pi,\theta) := \inf\left\{\frac{\mathcal{E}(f,f)}{\langle f,f\rangle_{L^2(\theta)}}:f\in \mathcal{D}(\mathcal{E}), f\not=0, f|_K=0\right\}.
\end{equation*}
\\
We conclude that $\lambda^{ess}_0(\Gamma,\pi,\theta) \leq I_\rho(\mu/2)$.
\end{proof}

Combining Lemma~\ref{thm1DF} with Theorem~\ref{endthm} proves Theorem~\ref{thm1}, and combining Lemma~\ref{thm2DF} with Theorem~\ref{endthm} proves Theorem~\ref{thm2}. \\

{\bf Remark:}  We briefly discuss the relevance of the condition \eqref{thetalb} to this part of the proof.  Under the assumption $\mu < \infty$, all balls have finite volume but may contain infinitely many points if \eqref{thetalb} does not hold.  Consequently, the function $g_k$ may not belong to $\mathcal{D}(\mathcal{E})$.  One may then wish to take the domain of our Dirichlet form to be the larger set $\mathcal{D}_{max}(\mathcal{E}) := \{f\in C(G):\mathcal{E}_1(f,f)<\infty\}$, and Lemma~\ref{lemDFg} shows that $g_k\in \mathcal{D}_{max}(\mathcal{E})$.  However, the Dirichlet form $(\mathcal{E},\mathcal{D}_{max}(\mathcal{E}))$ may then fail to be regular.

\section{Discreteness of the spectrum}

If the operator $\mathcal{L}_\theta$ is bounded on $L^2(\theta)$, then the essential spectrum is nonempty.  For graphs with the standard weights and the operator $\mathcal{L}_\pi$, Theorem 1 of \cite{Fu2} establishes necessary and sufficient condition for the essential spectrum to consist of a single point in terms of isoperimetric quantities. \\

If $\mathcal{L}_\theta$ is unbounded, it is possible for the essential spectrum to be empty (or, equivalently, for the spectrum to be discrete).  Graphs on which this phenomenon occurs (using the operator $\mathcal{L}_{\mathds{1}}$) are given in Examples 2 and 3 of Section 4.  By Corollary~\ref{cor2}, a necessary condition for absence of essential spectrum is that for any adapted metric $\rho$, $VG(\Gamma,\pi,\theta,\rho) = \infty$.  Examples 2 and 3 also show that for any $\varepsilon>0$, it is possible to find a graph $(\Gamma,\pi)$ and a metric $\rho$ adapted to $\mathcal{L}_{\mathds{1}}$ such that $VG(\Gamma,\pi,\mathds{1},\rho) = \infty$ and $\mathcal{L}_{\mathds{1}}$ has empty essential spectrum, but for every $x_0\in G$,

\begin{equation*}
\lim_{r\to\infty} \frac{1}{r^{1+\varepsilon}} \log V_\rho(x_0,r) = 0.
\end{equation*}
\\
In particular, in these examples the essential spectrum becomes empty precisely as the volume growth changes from exponential to superexponential.  Consequently, this condition may be optimal in applications. \\

There are connections with stochastic incompleteness of graphs.  In \cite{KLVW}, the following result is established:

\begin{thm}
(Keller-Lenz-Vogt-Wojciechowski, \cite{KLVW})
Let $(\Gamma,\pi,\theta)$ be a weakly spherically symmetric graph.  If $(\Gamma,\pi,\theta)$ is stochastically incomplete (i.e., if the continuous-time random walk associated with $\mathcal{L}_\theta$ explodes with positive probability), then $\lambda_0(\Gamma,\pi,\theta)>0$ and $\sigma_{ess}(\Gamma,\pi,\theta) = \varnothing$.
\end{thm}

In fact, their bound gives a quantitative estimate on the bottom of the spectrum in terms of certain volume growth quantities different than the ones considered in this paper. \\

Restricting for a moment to the setting of weakly spherically symmetric graphs, there is a large gap between the minimum possible volume growth for stochastic incompleteness and the minimum possible volume growth for discreteness of the spectrum.  In previous work of the author (see \cite{F2}), it was proven that stochastic completeness is implied by the estimate

\begin{equation*}
V_\rho(x_0,r) \leq Ce^{cr^2\log r}
\end{equation*}
\\
for some adapted metric $\rho$.  In general $r^2$ cannot be replaced with $r^{2+\varepsilon}$ and $\log r$ cannot be replaced with $(\log r)^{1+\varepsilon}$.  On the other hand, Examples 2 and 3 give examples where the essential spectrum is empty and

\begin{equation*}
V_\rho(x_0,r) \leq Ce^{cr^{1+\varepsilon}}.
\end{equation*}
\\
Consequently, the adapted volume growth thresholds for stochastic incompleteness and absence of essential spectrum are in general very different.  It would be interesting to know whether stochastic incompleteness is linked with the bottom of the spectrum being positive and the essential spectrum being empty for general weighted graphs.  However, it is not clear that volume growth techniques are relevant to this question, since there exist stochastically complete graphs with arbitrarily large volume growth.

\section{Examples}

{\bf 1.}  $k-$regular tree, $k\geq 3$.  We let $T_k$ denote the $k-$regular tree, which we equip with the standard weights.  In \cite{Br2}, \cite{Su}, it was proven that

\begin{equation} \label{botspec}
\lambda_0(T_k,\pi) = 1-\frac{2\sqrt{k-1}}{k}.
\end{equation}
\\
This graph satisfies $\mu_d(\pi) = \log(k-1)$.  Consequently, Theorem~\ref{thm1} yields

\begin{equation*}
\lambda^{ess}_0(\Gamma,\pi) \leq \frac{(1-\exp(\frac{1}{2}\mu_d(\pi)))^2}{1+\exp(\mu_d(\pi))} = 1-\frac{2\sqrt{k-1}}{k},
\end{equation*}
\\
so that there is equality in Theorem~\ref{thm1}.  This was observed earlier in \cite{Fu1}. \\

Similarly, for $\mathcal{L}_{\mathds{1}}$, $k-$regularity implies $k\mathcal{L}_{\pi} = \mathcal{L}_{\mathds{1}}$, and hence \eqref{botspec} implies

\begin{equation*}
\lambda_0(T_k,\mathds{1}) = k\left(1-\frac{2\sqrt{k-1}}{k}\right) = k-2\sqrt{k-1}.
\end{equation*}
\\
We compute the volume growth in the adapted metric $\frac{1}{\sqrt{k}}d$.  In this metric, two vertices at distance $r$ are at distance $\sqrt{k}r$ in the graph metric, and consequently $|B_{\frac{1}{\sqrt{k}}d}(x_0,R)| = \sum^{\sqrt{k}R}_{j=0} (k-1)^j$, from which it follows immediately that

\begin{equation*}
\mu_{\frac{1}{\sqrt{k}}d}(\mathds{1}) = \sqrt{k}\log(k-1).
\end{equation*}
\\
Again, Theorem~\ref{thm1} yields

\begin{equation*}
\lambda^{ess}_0(\Gamma,\mathds{1}) \leq k\cdot \frac{(1-\exp(\frac{1}{2\sqrt{k}}\mu_{\frac{1}{\sqrt{k}}d}(\mathds{1})))^2}{1+\exp(\frac{1}{\sqrt{k}}\mu_{\frac{1}{\sqrt{k}}d}(\mathds{1}))}  = k-2\sqrt{k-1},
\end{equation*}
\\
so that Theorem~\ref{thm1} may yield sharp results for the operator $\mathcal{L}_{\mathds{1}}$ also. \\

Note also that for large $k$, $\lambda_0(T_k,\mathds{1}) \geq \frac{k}{2}$, whereas $VG(T_k,\mathds{1},\frac{1}{\sqrt{k}}d) = \sqrt{k}\log(k-1)$.  Consequently, for any $\varepsilon\in (0,2)$ and any $C>0$, if one takes $K=K(\varepsilon,C)$ large enough, then

\begin{equation*}
\lambda_0(T_K,\mathds{1}) > C\left(VG(T_K,\mathds{1},\frac{1}{\sqrt{k}}d)\right)^{2-\varepsilon},
\end{equation*}
\\
so that Theorem~\ref{thm2} is asymptotically sharp. \\

{\bf 2.} Birth-death process:  Set $\Gamma := (\mathbb{Z}_+,E_{nn})$, where $E_{nn}$ denotes the set of nearest-neighbor edges, and set $\pi_\alpha(\{n,n+1\}) := (n+1)^2\log^\alpha_+(n+1)$ for $\alpha\in(-\infty,2)$, where $\log_+(x) := \log(x)\vee 1$. \\

Since $\pi_\alpha(B_d(0,r)) \leq Cr^{3+\varepsilon}$ for large $r$, it follows that $\mu_d(\pi_\alpha) = 0$ for any $\alpha\in (-\infty,2)$, and hence $\lambda^{ess}_0(\Gamma,\pi_\alpha,\pi_\alpha) = 0$ for all $\alpha\in (-\infty,2)$. \\
 
However, things are quite different for the unbounded operator $\mathcal{L}_{\mathds{1}}$.  In this setting, we use the adapted metric $\frac{1}{\sqrt{2}}d_E$.  As $R\to\infty$, we have

\begin{equation*}
\frac{1}{\sqrt{2}}d_E(0,R) = \frac{1}{\sqrt{2}}\sum^R_{j=1} \frac{1}{j\log_+^{\alpha/2}(j)} \sim \frac{\sqrt{2}}{2-\alpha}\log^{1-\alpha/2}(R),
\end{equation*}
\\
from which it follows that

\begin{equation*}
\log |B_{\frac{1}{\sqrt{2}}d_E}(0,r)| \asymp \left(\frac{2-\alpha}{\sqrt{2}}r\right)^{2/(2-\alpha)},
\end{equation*}
\\
and hence

\begin{equation*}
\mu_{\frac{1}{\sqrt{2}}d_E}(\mathds{1}) =
\begin{cases}
0 &\text{if }\alpha<0, \\
\sqrt{2} &\text{if }\alpha=0, \\
\infty &\text{if }0<\alpha<2. \\
\end{cases}
\end{equation*}
\\
We conclude that $\lambda^{ess}_0(\Gamma,\pi_\alpha,\mathds{1}) = 0$ for $\alpha<0$. \\

The case $\alpha=0$, which corresponds to exponential volume growth, is more subtle.  From Theorem~\ref{thm2}, we have 

\begin{equation*}
\lambda^{ess}_0(\Gamma,\pi_0,\mathds{1}) \leq \frac{1}{8}(\mu_{\frac{1}{\sqrt{2}}d_E}(\mathds{1}))^2 = \frac{1}{4}.
\end{equation*}

We will show that $\lambda^{ess}_0(\Gamma,\pi_0,\mathds{1}) \geq \lambda_0(\Gamma,\pi_0,\mathds{1})\geq \frac{1}{9}$.  To do this, it suffices to show that if $f\in C_c(\mathbb{Z}_+)$,

\begin{equation} \label{posspec1}
\sum_{n\in\mathbb{Z}_+} (n+1)^2(f(n)-f(n+1))^2 \geq \frac{1}{9}\sum_{n\in\mathbb{Z}_+} f^2(n).
\end{equation}
\\
For $n\in\mathbb{Z}_+$, we define $g\in C_c(\mathbb{Z}_+)$ by $g(n) := (n+1)(f(n)-f(n+1))$.  Since $f(n)\to 0$ as $n\to\infty$, we can reconstruct $f$ from $g$ as 

\begin{equation*}
f(n) = \sum_{j\geq n} \frac{1}{j+1}g(j);
\end{equation*}
\\
consequently, \eqref{posspec1} is equivalent to 

\begin{equation*}
9\sum_{n\in\mathbb{Z}_+} g^2(n) \geq \sum_{n\in\mathbb{Z}_+}\left(\sum_{j\geq n} \frac{1}{j+1}g(j)\right)^2,
\end{equation*}
\\
so it suffices to show that the operator $T:C_c(\mathbb{Z}_+)\to C_c(\mathbb{Z}_+)$ defined by

\begin{equation*}
(Tf)(n) := \sum_{j\geq n}\frac{1}{j+1}f(j)
\end{equation*}
\\
is bounded on $L^2(\mathds{1})$, with $\|T\|\leq 3$.  This is accomplished through an application of Schur's test.  Let $u(n) := (n+1)^{-1/2}$.  Then for $n\in\mathbb{Z}_+$, simple estimation using the integral test shows that 

\begin{align*}
(Tu)(n) &\leq 3u(n), \\ 
(T^*u)(n) &\leq 3u(n),
\end{align*}
\\
and hence Schur's test implies that $\|T\|\leq \sqrt{3\cdot 3} = 3$. \\

A similar calculation shows the absence of essential spectrum for $0<\alpha<2$.  In this case, we consider the graphs $(\Gamma_k,\pi_{\alpha,k})$, where $\Gamma_k$ is the subgraph of $\Gamma$ induced by the subset $\{j\in\mathbb{Z}_+:k\geq j\}$ and $\pi_{\alpha,k}$ are the restriction of the edge weights $\pi_\alpha$ to $\Gamma_k$.  Proceeding as above, we consider the operator

\begin{equation*}
(T_\alpha f)(n) := \sum_{j\geq n}\frac{1}{(j+1)\log^{\alpha/2}_+(j+1)}f(j),
\end{equation*}
\\
on $(\Gamma_k,\pi_{\alpha,k})$.  If $u_\alpha(n) := (n+1)\log^{\alpha/2}_+(n+1)$, then

\begin{align*}
(T_\alpha u)(n) &\leq \frac{3}{\log^{\alpha/2}_+(k+1)}u_\alpha(n), \\ 
(T^*_\alpha u)(n) &\leq \frac{3}{\log^{\alpha/2}_+(k+1)} u_\alpha(n),
\end{align*}
\\
so that on $(\Gamma_k,\pi_{\alpha,k})$, $\|T_\alpha\| \leq \frac{3}{\log^{\alpha/2}_+(k+1)}$, and consequently $\lambda_0(\Gamma_k,\pi_{\alpha,k},\mathds{1}) \geq \frac{\log^{\alpha}_+(k+1)}{9}$. \\

Since $\lambda^{ess}_0(\Gamma,\pi_\alpha,\mathds{1}) = \lim_{k\to\infty} \lambda_0(\Gamma_k,\pi_{\alpha,k},\mathds{1}) = \infty$, we conclude that the essential spectrum is empty for $0<\alpha<2$. \\

This example shows that even for graphs on which the VSRW behaves very differently from the CSRW (and, in particular, the operator $\mathcal{L}_{\mathds{1}}$ is unbounded), the critical behavior of the bottom of the spectrum for the operator $\mathds{1}$ is correctly predicted by computing the volume growth in a metric adapted to $\mathcal{L}_{\mathds{1}}$.  Moreover, the relevant volume growth estimate for unbounded operators, Theorem~\ref{thm2}, is analogous to Brooks' result for manifolds. \\

In addition to the present results on the spectra of weighted graphs, other phenomena for the operators $\mathcal{L}_\theta$ are correctly predicted by estimating volume growth in adapted metrics.  For example, in \cite{F2}, a volume growth criterion for stochastic completeness (or non-explosiveness) of the VSRW is established.  This result also employs adapted metrics, and (as with the present criteria) is analogous to volume growth criteria for manifolds.  When applied to this example, the results of \cite{F2} show that on $(\Gamma_\alpha,\pi_\alpha)$ the VSRW is non-explosive if $\alpha\leq 1$, and it is not difficult to show directly that the VSRW is explosive if $\alpha>1$. \\

It seems that the condition of adaptedness is in many respects the correct one for analyzing continuous-time simple random walks on graphs, especially if the corresponding graph Laplacian is unbounded.  These metrics allow one to prove various sharp estimates which are not possible to obtain using only the graph metric (such as heat kernel estimates and volume growth thresholds for various phenomena), and typical results using these metrics are directly analogous to previous results for manifolds or for metric measure spaces. These metrics are capable of detecting `fine' changes in geometry which are not visible to the graph metric (such as modifying the exponent on the logarithm in this example) but which may nevertheless have a profound effect on analytic properties of the Laplacian or the behavior of the associated random walk. \\

{\bf 3.}  Spherically symmetric trees with increasing rate of branching: \\

Let $\Gamma_\alpha$ be a tree rooted at $x_0$, with all vertices at a graph distance of $r$ from $x_0$ having $k(r):= \lfloor 2r^\alpha \rfloor$ neighbors at a graph distance of $r+1$ from $x_0$ for $0\leq \alpha<2$.  We equip these graphs with the standard weights, and consider the operator $\mathcal{L}_{\mathds{1}}$. \\

Using the adapted metric $d_V$, we have that if $d(x_0,x_R)=R$ and $0<\alpha<2$, then

\begin{equation*}
d_V(x_0,x_R) \asymp \sum^R_{j=1} \frac{1}{j^{\alpha/2}} \asymp R^{1-\alpha/2}.
\end{equation*}
\\
Consequently, if $d_V(x_0,y)\asymp r$, then $d(x_0,y)\asymp r^{2/(2-\alpha)}$.  As well, $|B_d(x_0,r)|=\sum^r_{j=0}\prod^{j}_{i=0}k(i)$, so that

\begin{align*}
\log |B_{d_V}(x_0,r)| \asymp \sum^{r^{2/(2-\alpha)}}_{j=1} \log j^\alpha \asymp \begin{cases} r &\text{if }\alpha=0, \\ r^{2/(2-\alpha)}\log r &\text{if }0<\alpha<2.\end{cases}
\end{align*}
\\ 
Since we have exponential volume growth for $\alpha=0$ in the $d_V$ metric, the essential spectrum of $\mathcal{L}_{\mathds{1}}$ on $\Gamma_0$ is nonempty. \\

For $0<\alpha<2$, since $\Gamma_\alpha$ has positive Cheeger constant, Theorem 2 of \cite{K} implies that the essential spectrum is empty.  In particular, for any $\varepsilon>0$ there exists some $\alpha_0\in (0,2)$ such that $\Gamma_{\alpha_0}$ satisfies $VG(\Gamma_{\alpha_0},\mathds{1},d_V) \leq Ce^{cr^{1+\varepsilon}}$ and the essential spectrum of $\mathcal{L}_{\mathds{1}}$ on $\Gamma_{\alpha_0}$ is nonempty. \\

Consequently, this example, as well as the results in the previous example on absence of essential spectrum, shows that Corollary~\ref{cor2} is optimal in certain settings. \\

{\bf Acknowledgements:}  This research was completed at the Statistical Laboratory at the University of Cambridge during a visit from by author.


\begin{thebibliography}{}

\bibitem{Br1} R. Brooks.  A relation between growth and the spectrum of the Laplacian, \emph{Math. Z.} {\bf 178} (1981), 501-508.

\bibitem{Br2} R. Brooks.  The spectral geometry of $k-$regular graphs, \emph{J. Anal. Math.} {\bf 57} (1991), 120-151.

\bibitem{D1}  E. B. Davies.  Analysis on graphs and noncommutative geometry, \emph{J. Funct. Anal.} {\bf 111} (1993), 398-430. 

\bibitem{D2}  E. B. Davies.  Large deviations for heat kernels on graphs, \emph{J. Lond. Math. Soc.} {\bf 47} (1993) 65-72.

\bibitem{DK} J. Dodziuk, L. Karp.  Spectral and function theory for combinatorial Laplacians, in ``Geometry of Random Motion'' Contemporary Mathematics 73, Amer. Math. Soc., 1988, 25-40.

\bibitem{F1} M. Folz.  Gaussian upper bounds for heat kernels of continuous time simple random walks on graphs.  \emph{Elec. J. Prob.} {\bf 62} (2011), 1693-1722.

\bibitem{F2} M. Folz.  Volume growth and stochastic completeness of graphs.  Preprint.

\bibitem{Fu1}  K. Fujiwara.  Growth and the spectrum of the Laplacian of an infinite graph, \emph{Tohoku. Math. J.} {\bf 48} (1996) 293-302.

\bibitem{Fu2}  K. Fujiwara.  The Laplacian on rapidly branching trees, \emph{Duke Math. J.} {\bf 83} (1996) 191-202.

\bibitem{FOT}  M. Fukushima, Y. Oshima, M. Takeda.  \emph{Dirichlet Forms and Symmetric Markov Processes}, de Gruyter, Berlin, 1994.

\bibitem{GHM}  A. Grigor'yan, X. Huang, J. Masamune.  On stochastic completeness of jump processes.  To appear in \emph{Math. Z.}

\bibitem{H}  X. Huang.  On uniqueness class for a heat equation on graphs.  Preprint.


\bibitem{K} M. Keller.  The essential spectrum of the Laplacian on rapidly branching tesselations.  \emph{Math. Ann.} {\bf 346} (2010), 51-66.

\bibitem{KL} M. Keller, D. Lenz.  Dirichlet forms and stochastic completeness of graphs and subgraphs.  To appear in \emph{J. Reine Angew. Math.}

\bibitem{KLW}  M. Keller, D. Lenz, R. Wojciechowski.  Volume growth, spectrum, and stochastic completeness of infinite graphs.  Preprint.  

\bibitem{KLVW}  M. Keller, D. Lenz, H. Vogt, R. Wojciechowski.  Note on basic features of large time behaviour of heat kernels.  Preprint. 

\bibitem{St} K.-T. Sturm.  Analysis on local Dirichlet spaces I.  Recurrence, conservativeness and $L^p$-Liouville properties.  \emph{J. Reine. Angew. Math.} {\bf 456} (1994), 173-196. 

\bibitem{Su}  T. Sunada.  Fundamental groups and Laplacians, \emph{Lecture Notes in Mathematics 1339}, Springer-Verlag, New York, 1988, 248-277.

\end{thebibliography}
\end{document}